\documentclass[conference]{IEEEtran}
\IEEEoverridecommandlockouts
\usepackage{amsmath,amssymb,amsfonts}
\usepackage{algorithmic}
\usepackage{graphicx}
\usepackage{textcomp}
\usepackage{xcolor}
\def\BibTeX{{\rm B\kern-.05em{\sc i\kern-.025em b}\kern-.08em
T\kern-.1667em\lower.7ex\hbox{E}\kern-.125emX}}

\usepackage[backend=biber, style=numeric]{biblatex}
\usepackage{amsmath, amsthm, amssymb, latexsym,epsfig,amsthm,enumerate,multicol,wasysym}
\usepackage{color}
\usepackage{graphicx}
\usepackage{nccrules}
\usepackage{xfrac}
\usepackage{hyperref}
\usepackage{textcomp}
\usepackage{tikz}
\usepackage{xcolor}
\usepackage{comment}
\usepackage{enumitem}
\usepackage{wrapfig}

\def\R{\mathbb{R}}

\newtheorem{theorem}{Theorem}[section]
\newtheorem{lemma}[theorem]{Lemma}

\newtheorem{proposition}[theorem]{Proposition}

\newtheorem{definition}[theorem]{Definition}

\newtheorem{case[theorem]}{Case}

\numberwithin{equation}{section}

\usepackage[normalem]{ulem}

\newcommand{\Z}{\mathbb Z}
\newcommand{\N}{\mathbb N}

\newcommand{\cW}{\mathcal W}

\newcommand{\cB}{\mathcal B}

\newcommand{\cI}{\mathcal I}

\newcommand{\cT}{\mathcal T}

\newcommand{\cF}{\mathcal F}

\newcommand{\bL}{\mathbf L}
\newcommand{\bk}{\mathbf k}

\def\R{\mathbb{R}}

\begin{document}

\title{Wave packet systems and connections to spectral analysis  of limiting operators 

\thanks{A.~Israel was supported by the Air Force Office of Scientific Research, under award FA9550-19-1-0005.  A.~Mayeli was supported in part by AMS-Simons Research Enhancement Grant and the PSC-CUNY research grants.}
}

\author{\IEEEauthorblockN{Kevin Hughes}
\IEEEauthorblockA{\textit{Department of Mathematics} \\
\textit{Edinburgh Napier University}\\
Edinburgh, UK \\
khughes.math@gmail.com}
\and
\IEEEauthorblockN{Arie Israel}
\IEEEauthorblockA{\textit{Department of Mathematics} \\
\textit{University of Texas at Austin}\\
Austin, Texas, USA \\
arie@math.utexas.edu}
\and
\IEEEauthorblockN{Azita Mayeli}
\IEEEauthorblockA{\textit{Department of Mathematics} \\
\textit{City University of New York}\\
New York, USA \\
amayeli@gc.cuny.edu}
}

\maketitle

\begin{abstract} 
We discuss the design of ``wave packet systems'' that admit strong concentration properties in phase space. We make a connection between this problem and topics in signal processing related to the spectral behavior of spatial and frequency-limiting operators. The results have engineering applications in medical imaging, geophysics, and astronomy. 
\end{abstract}

\begin{IEEEkeywords}
wave packets,   limiting operators, eigenvalue distributions
\end{IEEEkeywords}

%\tableofcontents 
 
\section{Introduction}

We write $\widehat{f}(\xi) = \int_{\R^d} f(x) \exp(- i x\cdot\xi ) dx$ for the Fourier transform of $f \in L^2(\R^d)$.  
% We equip $L^2(\R^d)$ with norm $\| f \| = \left( \int_{\R^d} |f(x)|^2 dx \right)^{1/2}$ and inner product $\langle f,g \rangle = \int_{\R^d} f(x) \overline{g(x)} dx$. 
%
Given a measurable set $S \subset \R^d$, we denote $\mathcal{PW}(S)$ for the Paley-Wiener space of $S$, which consists of all $f \in L^2(\R^d)$ such that $\widehat{f}(\xi) = 0$ for a.e. $\xi \in \R^d \setminus S$. Functions in $\mathcal{PW}(S)$ have  Fourier transforms supported on $S$. Consequently, we refer to functions $f \in \mathcal{PW}(S)$ as $S$-bandlimited.

If $S$ is bounded, functions $f(x)$ in $\mathcal{PW}(S)$ depend analytically on $x$; hence, they cannot be compactly supported.  However, one can ask to what extent a bandlimited function $f\in  \mathcal{PW}(S)$ can be concentrated on a given bounded set $F \subset \R^d$. This question can be characterized by the behavior of the eigenvalues of a certain compact operator on $L^2(\R^d)$.

Given bounded measurable sets $F$ and $S$ in $\R^d$ of positive Lebesgue measure, let $P_F: L^2(\R^d) \rightarrow L^2(\R^d) $ be the orthogonal projection on the subspace $L^2(F)$ of all $L^2$ functions supported on $F$, and let $B_S: L^2(\R^d) \rightarrow L^2(\R^d)$ be the orthogonal projection on % the Paley Wiener 
%space 
$\mathcal{PW}(S)$. A \emph{limiting operator} is  given by alternating these projection operators as 
\begin{align}\label{SSLO}
\cT_{F,S} = B_S P_F B_S.
\end{align}
This operator is self-adjoint, positive definite,  and compact, and possesses a decreasing 
 sequence of eigenvalues $\{\lambda_k(F,S)\}$ in $(0,1)$  
 % $$1>\lambda_1(F,S) \geq  \lambda_2(F,S)\geq \cdots >0$$
satisfying $\lambda_k(F,S)\to 0$ as $k\to \infty$. 
 % \footnote{The operators $B_SP_FB_S$  are  $P_FB_SP_F$ have identical eigenvalues with the same multiplicity -- see \cite{israelmayeli2023acha}.}

The eigenfunctions of $\mathcal{T}_{F,S}$ form an orthonormal basis $\{\Psi_k = \Psi_k(F,S) : k \geq 1\}$ for \( \mathcal{PW}(S) \), and they satisfy a double orthogonality property -- specifically, they are orthogonal with respect to the \( L^2(\mathbb{R}^d) \) inner product, and also, their restrictions to $F$ are orthogonal with respect to the $L^2(F)$ inner product \cite{Bell4}. This makes them an important tool in many practical problems, e.g., in the development of stable extrapolation methods for bandlimited functions (see \cite{landau1985overview}).

The eigenfunctions associated with the largest eigenvalues of $\mathcal{T}_{F,S}$, along with any function in the subspace spanned by these eigenfunctions, are bandlimited and exhibit the optimal energy concentration on the domain \( F \). Indeed, by the variational characterization of eigenvalues, the subspace $V_k := \mathrm{Span}\{\Psi_1,\cdots,\Psi_k\} \subset \mathcal{PW}(S)$ and the eigenvalue $\lambda_k = \lambda_k(F,S)$ solve the following optimization problem:
$$ \lambda_k = \max_{V \subset \mathcal{PW}(S)} \min_{f \in V \setminus \{0\}} \left\{ \int_F |f(x)|^2 dx/\|f\|^2    :  \dim V = k \right\}.
$$

The eigenfunctions $\Psi_k$  and associated subspaces $V_k$ play a major role in applications such as imaging, where the shapes of $F$ and $S$ are dictated by various acquisition constraints \cite{Bell1,Bell2,Bell3,Bell4,Bell5}.  
% They also provide a method for constructing finite-dimensional subspaces of the Paley-Wiener space, where each function in that subspace is bandlimited and has high energy in the $L^2$-norm on the spatial domain. 
For practical reasons, such as in thresholding and denoising, one needs to determine the dimension of such subspaces.  Thus, counting the eigenvalues \(\lambda_k(F,S)\) near $1$ and those far from both $1$ and $0$ is essential.
Generally, the solution to this problem depends on the geometric properties of the space and frequency domains.

In one-dimension, the eigenfunctions of the limiting operator $\mathcal{T}_{I,J}$ for intervals $I,J \subset \R$ are known as prolate spheroidal wave  functions,
%, and have many practical applications. 
and the eigenvalues $\lambda_k(I,J)$ depend only on the product of the time-frequency bandwidth, $c = |I| \cdot |J|$. 
As such, we represent these eigenvalues as $\lambda_k(c)$.
In \cite{Bell1,Bell2,Bell3}, it was established that these eigenvalues satisfy a strong clustering property: 
For any $\epsilon \in (0,1/2)$, 
approximately $c/2\pi - O_\epsilon(\log(c))$ eigenvalues lie in $[1-\epsilon, 1)$, 
fewer than $O_\epsilon(\log(c))$ many eigenvalues lie in the ``plunge region'' $(\epsilon ,1 - \epsilon) $,
and the remaining eigenvalues in $(0,\epsilon]$ decay to $0$ exponentially fast. 
Furthermore, it is known that the sequence $\lambda_k(c)$ crosses $\lambda=1/2$ when $k$ differs from the floor or ceiling of $c/2\pi$ by at most $1$ \cite{Landau93}. Therefore, there is a phase transition at $k = c/2\pi$: For $k\lesssim c/2\pi$ the eigenvalues $\lambda_k(c)$ are exponentially close to $1$ (for $c$ large), and for $k \gtrsim c/2\pi$, $\lambda_k(c)$ is exponentially close to $0$. When $k \sim c/2\pi$, the eigenvalues have intermediate size between $0$ and $1$. Recent advancements have yielded sharp quantitative bounds, making these statements precise -- see \cite{israel15eigenvalue,osipov2013,davenport2021improved, Bonami21, Kulikov24}.

%, as we discuss below.
Let $Q = [0,1]^d$ be the unit hypercube, and let $S\subset \R^d$ be a convex domain. We suppose $S$ is \emph{coordinate-wise symmetric} -- that is,  $\tau_j(x) \in S$ for every $x \in S$ and any coordinate reflection $\tau_j : \R^d \rightarrow \R^d$, $1 \leq j \leq d$, $\tau_j(x_1,\cdots,x_j,\cdots, x_d) = (x_1,\cdots,-x_j,\cdots,x_d)$. 
% 
% satisfying a symmetry condition. We say that a convex set $S \subset \R^d$ is \emph{coordinate-wise symmetric}  
%  provided that
% \begin{equation}\label{eqn:coord_sym}
%  (x_1,\cdots,x_d) \in S, \; \sigma = (\sigma_1,\cdots,\sigma_d) \in \{\pm 1\}^d \implies (\sigma_1 x_1,\cdots, \sigma_d x_d) \in S.
% \end{equation}
We write $B(x,r) \subset \R^d$ for the Euclidean ball centered at $x \in \R^d$ of radius $r > 0$. We write $S(r) := \{ rx : x \in S\}$ for the $r$-dilate of $S$.

The following theorem appeared before in \cite{israelmayeli2023acha}. 

\begin{theorem}\label{mainthm:cube_convex}
Let $Q = [0,1]^d$, and  let $S \subset B(0,1)$ be a compact coordinate-wise symmetric convex set in $\R^d$. 

There exists a dimensional constant $C_d > 0$ (independent of $S$) such that, for any $\epsilon \in (0,1/2)$ and $r \geq 1$, 
\begin{align}\label{eig_clust:eqn1}
% &| \# \{ k : \lambda_k(Q,S(r)) > \epsilon\} - (2\pi)^{-d} \mu_d(S(r)) | \leq C_d E_d(\epsilon,r) ,\\
%\label{eig_clust:eqn2}
& \# \{ k : \lambda_k(Q,S(r)) \in (\epsilon, 1-\epsilon)\} \leq C_d E_d(\epsilon,r),
\end{align}
where $E_d(\epsilon,r) := \max \{ r^{d-1} \log(r/\epsilon)^{5/2}, \log(r/\epsilon)^{5d/2} \}$.
\end{theorem}

In the following section, we introduce wave packet systems, generalizing the classical Gabor and wavelet systems, and discuss the concentration and packing properties for families of functions in  $L^2(\R^d)$. We then use wave packets to construct an orthonormal basis for $L^2([0,1]^d)$  with strong concentration properties. We conclude with the proof of Theorem \ref{mainthm:cube_convex}.

\section{Methodology and Main Results}
\subsection{Wave Packet Systems}

%%%%Discussion of signal decomposition and time-frequency analysis%%%%
A classical technique in signal analysis is the decomposition of signals into components which are localized in the physical (time) and frequency domains. For the analysis, one is interested to construct orthogonal systems (or more generally, frames) that admit strong concentration properties. Two commonly employed methods for signal analysis are the short-time Fourier transform (STFT) and wavelet transform. Related to these transforms are the associated notion of a Gabor system or wavelet system.
%%%%%%%%%%%%%%%%%%%%

%%%%Defn of Frame%%%%
\begin{definition}
A \emph{frame} for a closed subspace $H \subset L^2(\R^d)$ is a family of functions $\{\psi_\nu\}$ in $H$ satisfying the condition:
\[
A \| f \|^2 \leq \sum_\nu |\langle f, \psi_\nu \rangle|^2 \leq B\| f \|^2 \quad \forall f \in H.
\]
We refer to the constants $0 < A\leq B< \infty$ as the (lower and upper) \emph{frame bounds} for $\{\psi_\nu\}$. A \emph{unit-norm frame} also satisfies $\| \psi_\nu \| = 1$. 
\end{definition}

The frame condition asserts that $\{\psi_\nu\}$ is a redundant (overcomplete) spanning set for $H$.
%%%%%%%%%%%%%%%%%%%%

%%%%Defn of Gabor System%%%%
Let \( g \in L^2(\mathbb{R}^d) \) be a smooth window function, decaying rapidly at $\infty$, and let $\mathcal{S} = \mathcal{S}_1 \times \mathcal{S}_2 \subset \R^d \times \R^d$, $\mathcal{S}_1 = \{x_\nu\}_{\nu \in \mathcal{I}}$, $\mathcal{S}_2=\{\xi_\mu\}_{\mu \in \mathcal{J}}$ be a set of spatial translation and frequency modulation pairs. A \emph{Gabor system} generated by $g$ and $\mathcal{S}$ is given by applying the translations and modulations to the window function,
\begin{equation}\label{eq:GaborSystem}
\mathcal{G}(g,\mathcal{S}) := \Big\{ g_{\nu\mu}(x) = e^{ i(x - x_\nu)\cdot \xi_\mu} g(x - x_\nu) \Big\}_{\nu \in \mathcal{I}, \mu \in \mathcal{J}}.  
\end{equation} 
In the literature, Gabor systems are frequently studied in the special case where $\mathcal{S}_1$ and $\mathcal{S}_2$ are lattices in $\R^d$ \cite{grochenig2001foundations}.

If $\mathcal{S}$ is structured appropriately (often as a Cartesian product of sufficiently dense lattices), the system $\mathcal{G}(g,\mathcal{S})$ is complete in $L^2(\R^d)$, and even forms a frame. One should regard $g_{\nu\mu}$ as localized to the region $|x-x_\nu| < 1$ in physical domain and to the region $|\xi - \xi_\mu| < 1$ in the frequency domain. 
%(This can be made precise by enforcing decay conditions on the window function $g$ and its Fourier transform $\widehat{g}$.) 
Gabor systems are widely used in signal processing to perform localization of signals. Given a signal $f \in L^2(\R^d)$, the inner products $\{ \langle f, g_{\nu\mu} \rangle\}$ capture the local spatial and frequency content of the signal near the set of points $\{(x_\nu,\xi_\mu)\}$ in phase space $\R^d \times \R^d$. If the system $\{g_{\nu\mu}\}$ forms a frame in $L^2(\R^d)$, the signal $f$ can then be reconstructed from the inner products $\{\langle f,g_{\nu\mu}\rangle\}$. This recovery problem is solved using the synthesis operator, defined via a suitable dual frame. Therefore, Gabor systems provide a convenient tool for performing analog-to-digital conversion.
%%%%%%%%%%%%%%%%%%%%

Another example of localization arises in the wavelet representation of a function. Starting with a function $g$, often called a ``mother wavelet'', one defines a family given by applying a family of translations and dyadic dilations via
\begin{align}\label{dyadic-wavelet-system}
\mathcal{W}(g) := 
% \{w_{j,\vec{k}}(x) = 2^{-jd/2} g(2^{-j}x - \vec{k}) :j \in \Z, \vec{k} \in \Z^d \}
\big\{ w_{j,\vec{k}}(x) = 2^{-jd/2} g(2^{-j}x - \vec{k}) \big\}_{j \in \Z, \vec{k} \in \Z^d}.
\end{align}
Under suitable conditions on $g$, this system $\mathcal{W}(g)$ can be shown to be an orthonormal basis for a subspace of $L^2(\R^d)$. One should regard the function $w_{j,\vec{k}}$ as localized to a region $2^{-j-1} <|\xi| < 2^{-j}$ in the frequency domain, and to the region $|x-\vec{k} 2^{j}| < 2^j$ in the physical domain. Given a signal $f$, the inner products $\langle f, w_{j,\vec{k}})$ measure the content of $f$ within the corresponding region of the phase space $\R^d \times \R^d$.

We now introduce the notion of a \emph{wave packet system}, unifying the previous examples. We let $\Theta = \{\theta_\nu : \R^d \rightarrow [0,\infty) \}_{\nu \in \mathcal{I}}$ be a family of $C^\infty$ functions satisfying a uniform decay estimate, $|\theta_\nu(x)| \leq C_d/(1+|x|^d)$, with $C_d$ independent of $\nu$, and normalized so that $\| \theta_\nu \| = 1$. Fix a family of affine maps $\mathcal{A} = \{T_\nu : \R^n \rightarrow \R^n\}$, with $T_\nu(x) = A_\nu(x-x_\nu)$ for $x_\nu \in \R^n$ and $A_\nu \in GL_n(\R)$. Fix a family of frequency modulation vectors $\Xi=\{\xi_\mu \in \R^n\}_{\mu \in \mathcal{J}}$. Define a \emph{wave packet system} associated with these families as
\begin{equation}
\mathcal{P}(\Theta,\mathcal{A},\Xi) = \{ g_{\nu\mu}(x) = c_\nu e^{ i T_\nu(x) \cdot \xi_\mu}\theta_\nu(T_\nu(x))  \},
\end{equation}
with $c_\nu = |\det A_\nu|^{1/2}$ picked so that $\| g_{\nu\mu}\| = 1$. Often, it is convenient to let $\theta_\nu = \theta$ be independent of $\nu$. 
% When the affine maps \( T_\nu \) are simply translations, given by \( x \mapsto x - x_\nu \), the wave packet system becomes a special case of the Gabor system~\eqref{eq:GaborSystem}.
Notice that Gabor systems~\eqref{eq:GaborSystem}  are a special case of wave packets systems, where the maps $T_\nu$ are simply translations, $x \mapsto x - x_\nu$. 
Also, wavelet systems \eqref{dyadic-wavelet-system} are a special case of wave packet systems in which the affine maps $T_\nu$ specify a family of dyadic dilations and translations along dyadic lattice points, and there is no frequency modulation, $\Xi = \{0\}$. In general, a wave packet $g_{\nu \mu}$ should be regarded as localized to an ellipsoidal region $A_\nu^{-1}(B(0,1)) + x_\nu$ in physical space, and localized to the dual region $A_\nu^{T} (B(0,1)) + \xi_\mu$ in frequency space. 

%In the following section, we shall construct a family of functions $\{ \psi_\nu(x)\}$, inspired by the above (with one small tweak), involving the action of the affine group and frequency modulation, which forms an orthonormal basis for the Hilbert space $L^2([0,1]^d)$. Our construction is related to the Coifman-meyer local cosine basis \cite{CM}. We apply this basis to prove bounds on the number of eigenvalues of the limiting operator \eqref{SSLO} that are away from zero. This result is presented in Theorem~\ref{}.

We end the section by stating several problems related to the concentration properties of systems of functions in $\R^d$. For this discussion, we fix domains $F, S \subset \R^d$. We will view $F \times S$ as a domain in the phase space $\R^d \times \R^d$. Phase space will be parametrized by $(x,\xi) \in \R^d \times \R^d$. 
%The classical uncertainty principle in Fourier analysis states that any nonzero function $\psi \in L^2(\R^d)$ cannot be compactly supported while having a compactly supported Fourier transform. Further variants of the uncertainty principle assert that if a function $\psi \in L^2(\R^d)$ is limited to a region $B \subset \R^d$ of diameter $\Delta$ then its Fourier transform cannot be ``well-concentrated'' on a region of diameter at most $\Delta^{-1}$ -- this rough statement can be made precise in various ways, e.g. see \cite{}. 
We are interested in constructing a system of functions \( \{\psi_\nu\}_{\nu \in \cI} \) in \( L^2(\mathbb{R}^d) \) that are highly concentrated in the spatial and frequency domains \( F \) and \( S \) (although not necessarily perfectly supported on either). The following definition makes this precise.

\begin{definition}[$\epsilon$-concentrated system]
A finite 
 system of unit-norm functions $\{\psi_\nu\}$ in $L^2(\R^d)$ is $\epsilon$-concentrated on $F \times S$ provided
\[
\sum_\nu \| \psi_\nu \|_{L^2(\R^d \setminus F)}^2 + \sum_\nu \| \widehat{\psi_\nu}\|_{L^2(\R^d\setminus S)}^2 \leq \epsilon^2.
\]
\end{definition}

We note that the previous definition imposes no orthogonality constraints on \(\{\psi_\nu\}\).

\begin{definition}[$\epsilon$-packing]
Given compact sets $F,S \subset \R^d$, given $\epsilon > 0$, we say that  a finite family of unit-norm functions $\mathcal{F} = \{\psi_\nu\}_{\nu\in \cI}$ in $L^2(\R^d)$ is  an $\epsilon$-packing for $F \times S$ if
\begin{enumerate}
    \item $\{\psi_\nu\}$ is an  $\epsilon$-concentrated on $F \times S$. 
    \item $|\langle \psi_\nu, \psi_{\nu'} \rangle | < \epsilon$ for all $\nu,\nu'$, $\nu\neq \nu'$.
    % \item $\| \psi_\nu \| = 1$ for all $\nu$.
\end{enumerate}
\end{definition}

% \textbf{Packing problem:} Given compact sets $F,S \subset \R^d$, given numbers $\epsilon > 0$, and $\delta > 0$, we say that a family $\{\psi_\nu : \nu \in \mathcal{I}$ of $L^2(\R^d)$ functions is an $(\epsilon,\delta)$-packing for $F \times S$ provided that $|\langle \psi_\nu, \psi_{\nu'} \rangle | < \delta$ for all $\nu,\nu'$ \AM{disjoint?}, and the family $\{\psi_\nu\}$ is $\epsilon$-concentrated on $F \times S$. 
Condition 2 in the definition of a packing asserts that the family $\{\psi_\nu\}$ is approximately orthogonal.
Of course, every domain $F \times S$ admits the empty packing, 
 $\mathcal{F} = \emptyset$. In the previous definition, we would like to construct packings with large cardinality $\#\cF$. We refer to this as {\it the packing problem}. Ideally, we seek bounds on $\# \cF $  in terms of $\epsilon$ and the Lebesgue measures of $F$ and $S$.

The following lemma is essentially proven in \cite{Kulikov24} -- see the discussion in Section 4. It connects the packing problem to lower bounds on the eigenvalues of limiting operators. We provide the proof here to illustrate the variational technique.

% The packing problem is connected to lower bounds on the eigenvalues of limiting operators due to the following lemma.
\begin{lemma}\label{packing_spectral:lem}
Let $\cF = \{\psi_\nu\}$ be an $\epsilon$-packing  for $F \times S$, with $n := \#\cF$, and $\epsilon < \frac{1}{2n}$. Then $\lambda_{n}(P_F B_S P_F) > 1 - 5 \epsilon n^{1/2}$.
\end{lemma}
\begin{proof}
We write $\cF = \{\psi_1,\cdots,\psi_n\}$, and let $V \subset L^2(\R^d)$ be the subspace spanned by $\cF$. We claim that $\cF$ is a linearly independent set. Indeed, let $G$ be the $n \times n$ Gram matrix of $\cF$, i.e., $G_{\nu \nu'} = \langle \psi_\nu,\psi_{\nu'}\rangle$, $1 \leq \nu,\nu' \leq n$. The diagonal entries of $G$ are $1$, and the off-diagonal entries of $G$ have magnitude at most $\epsilon$. Let $I$ denote the $n \times n$ identity matrix. Observe that the Frobenius norm of $I - G$ is at most $\epsilon n < 1/2$. The spectral norm is bounded by the Frobenius norm, and so, $\| I - G \| < 1/2$, hence $G$ is invertible, and so $\cF$ is linearly independent. Thus, $\dim V = n$. 

By the variational characterization of eigenvalues (the max-min theorem), the bound $\lambda_n(F,S) > 1 - C \epsilon n^{1/2}$ is implied by
% \begin{align}
%  \|P_F B_S P_F \psi\|  > (1- C \epsilon n) \| \psi \| \quad \forall \psi \in V,  ~ \psi \neq 0.
% \end{align}
\[
\frac{\|P_F B_S P_F \psi\|}{\| \psi \|} > 1- C \epsilon n^{1/2} \quad \forall \psi \in V,  ~ \psi \neq 0.
\]

The claim that $\{\psi_\nu\}$ is $\epsilon$-concentrated on $F \times S$ implies that $\|(Id - P_F) \psi_\nu \| \leq \epsilon$ and $\|(Id - B_S) \psi_\nu \| \leq \epsilon$ for all $\nu$. Observe that $Id - P_F B_S P_F = (Id - P_F) + P_F(Id - B_S) + P_F B_S (Id - P_F)$. Since $P_F$ and $B_S$ are projections, by the triangle inequality we obtain $\|(Id - P_F B_S P_F) \psi_\nu \| \leq 3 \epsilon$.

Let $\psi \in V$, and write $\psi = \sum_\nu a_\nu \psi_\nu$. By the triangle inequality and the estimation in the previous paragraph, we have $\| (Id - P_FB_SP_F) \psi \| \leq 3 \epsilon \sum_{\nu} |a_\nu|$, which implies that
\begin{align*}
\| P_F B_S P_F \psi \| \geq \| \psi \| - 3 \epsilon \sum_\nu |a_\nu| \geq \| \psi\| - 3 \epsilon n^{1/2} 
 % \left(\sum_\nu |a_\nu|^2\right)^{1/2}.
 \sqrt{\sum_\nu |a_\nu|^2}.
\end{align*}
 
By expanding out the inner product, we write
\[
\| \psi \|^2 = \sum_{\nu} |a_\nu|^2  + \sum_{\nu\neq \mu} a_\nu \overline{a_\mu} \langle \psi_\nu, \psi_\mu \rangle \geq (1 - n\epsilon) \sum_\nu |a_\nu|^2,
\]
where in the last inequality we used $\left|\sum_{\nu \neq \mu} a_\nu \overline{a_\mu} \right| \leq (\sum |a_\nu| )^2 \leq n \sum_\nu |a_\nu|^2$, as well as the approximate orthogonality of $\cF$. Combining the previous estimates,
\[
\frac{\| P_F B_S P_F \psi \|}{\| \psi \|}  \geq 1 - 3 \epsilon n^{1/2}(1-n\epsilon)^{-1/2} \geq 1 - 5 \epsilon n^{1/2}.
\]
This estimate holds for any nonzero $\psi \in V$. This concludes the proof of the lemma.
\end{proof}

In dimension $d=1$, when $I,J \subset \R$ are intervals, \cite{Kulikov24} gives the construction of packings of $I \times J$ using the time-frequency translates of a family of Hermite polynomials multiplied by a Gaussian window. To state their result, fix $c:= |I| \cdot |J|$ to be sufficiently large. 
Given $\delta < 1$, there exists $\alpha(\delta) < 1$ and an $\epsilon$-packing $\cF$ of $I \times J$ such that $\epsilon = \alpha(\delta)^c$ and $\#\cF \geq (1-\delta) c/2 \pi$. 
%It would be very interesting to see whether a similar result holds for a wave packet system $\cF$, consisting of time-frequency translations of the Gaussian window.

As noted in \cite{Kulikov24}, these results and Lemma \ref{packing_spectral:lem} imply that for any $\delta > 0$,  there exists $\alpha = \alpha(\delta) < 1$, so that, for $c = |I| \cdot |J| $,
\[
\lambda_n(\mathcal{T}_{I,J}) = \lambda_n(P_I B_J P_I) \geq 1 - 5 \alpha^c c^{1/2} \quad \forall \; n \leq (1-\delta) c/2\pi.
\]
This gives the sharpest known lower bounds on the eigenvalues of time-frequency limiting operators in one-dimension, before the ``plunge region'' (i.e., for $n < c/2\pi$).

\textbf{Open Problem 1:} Let $I,J \subset \R$ be intervals with \(c:=|I|\cdot|J|\) sufficiently large. Given $\delta <1$, does there exist a wave packet system $\cF$ in $L^2(\R)$, given as a family of time-frequency translates of a \emph{single} smooth window function,  that is an $\epsilon$-packing for $I \times J$ with $\# \cF \geq (1-\delta) c/2\pi$ and $\epsilon = \alpha^c$ for $\alpha = \alpha(\delta) < 1$? Further, can this result be extended to higher dimensions, for the packing of a Cartesian product of Euclidean balls?

The next lemma is the key tool used in the proof of Theorem \ref{mainthm:cube_convex}. It connects the existence of a concentrated system with a bound on the set of the eigenvalues of a limiting operator.

\begin{lemma}\label{conc_frame:lem} Suppose there exists a unit-norm frame 
$\mathcal{F} = \{\psi_\nu\}$ 
% $\mathcal{F} = \{\psi_\nu : \nu \in \cI\}$ 
 for $L^2(F)$ with frame constants $0<A \leq B < \infty$, that admits a partition into three subfamilies, $\mathcal{F} = \mathcal{F}_{low} \cup \mathcal{F}_{hi} \cup \mathcal{F}_{res}$, such that $\mathcal{F}_{low}$ is $\sqrt{A/4}\cdot\epsilon$-concentrated on $F \times S$, $\mathcal{F}_{hi}$ is $\sqrt{A/4}\cdot\epsilon$-concentrated on $F \times (\R^d \setminus S)$, and   the family $\cF_{res}$ is finite. 
Then the set of eigenvalues of the operator $P_F B_S P_F$ satisfies the following spectral bound: $$\# \{ \lambda_k(P_F B_S P_F) \in (\epsilon,1-\epsilon) \} \leq (2/A) \# \mathcal{F}_{res}. $$
\end{lemma}

This lemma is proven using the variational principle, and variants of this lemma appeared previously in the literature -- see Lemma 1 of \cite{israel15eigenvalue} for a related result when $\cF$ is an orthonormal system. This result, for general frames, is implied by Lemma 3.1 in \cite{marceca2023} (we present the reduction below). 

\begin{proof}[Proof of Lemma \ref{conc_frame:lem}]
Since $\cF \subset L^2(F)$ we have $P_F \psi = \psi$ for all $\psi \in \cF$. Since $P_F$ is a projection, it follows that
\[
\| (Id - P_F B_S P_F) \psi \| = \| P_F(Id - B_S) \psi \| \leq \|(Id - B_S) \psi \|
\]
and $\|P_F B_S P_F \psi \| \leq \| B_S \psi \|$.
Therefore, by Plancherel's theorem, and the concentration properties of $\cF_{low}$ and $\cF_{hi}$, 
\begin{align*}
&\sum_{\psi \in \cF_{low}} \| (Id -P_FB_SP_F) \psi \|^2 + \sum_{\psi \in \cF_{hi}} \| P_FB_SP_F \psi \|^2 \\
& \leq \sum_{\psi \in \cF_{low}} \| \widehat{\psi} \|_{L^2(\R^d \setminus S)}^2 + \sum_{\psi \in \cF_{hi}} \| \widehat{\psi} \|_{L^2( S)}^2 \leq \frac{A}{2} \epsilon^2.
\end{align*}
The result follows then by Lemma 3.1 of \cite{marceca2023}.
\end{proof}

\subsection{Concentrated wave packet systems on the unit cube.}

\subsubsection{Coifman-Meyer local sine basis}

We introduce an orthonormal basis for $L^2([0,1])$, consisting of compactly supported functions whose Fourier transforms have rapidly decaying tails. Here, we will use a variant of the \emph{local sine basis} of Coifman-Meyer \cite{CM}.
 
% We use $\omega$ to denote the frequency variable in $\R$, and let $\mathcal{F}(f)(\omega)  = \widehat{f}(\omega) = \int_\R e^{- i x \omega} f(x) dx$ be the Fourier transform of a function $f \in L^2(\R)$.

Define $\cW$ to be a realization of the \emph{Whitney decomposition} of $(0,1)$. That is, the family $\cW$ consists of intervals, forming a partition of $(0,1)$, given by
\[
\cW =  \{ [ 2^{-j-1}, 2^{-j} ) \}_{j \geq 1} \cup  \{ [ 1 - 2^{-j}, 1 - 2^{-j-1})\}_{j \geq 1}.
\]
Let $\delta_L \in (0,1)$ denote the length of an interval $L \in \cW$. 
%For $\delta > 0$,
%\begin{equation}\label{good_geom1:eqn}
%\delta_L \leq \delta \implies L \subset [0,2 \delta] \cup [ 1 - 2 \delta,1]. 
%\end{equation}
% 

For $a>0$, define the ``envelope  function'' $\Psi_a(\xi) := \exp (- a \cdot \lvert \xi \rvert^{2/3})$. The main result of the section follows: 

\begin{proposition}\label{basis:prop}
There exists an orthonormal basis $\cB^1 = \{ \phi_{L,k} \}_{(L,k) \in \cW \times \N}$ for $L^2([0,1])$, with each $\phi_{L,k}$  a  $C^\infty$ function supported on a neighborhood of $L$, satisfying the Fourier decay condition: for all $\xi \in \R$, and some constants $C,a > 0$,
\begin{equation}
\label{Fourier_decay_CM:eqn}
\lvert \widehat{\phi_{L,k}} (\xi) \rvert \leq C \sqrt{\delta_L}  \sum_{\sigma \in \{\pm 1\}} \Psi_a\left( \delta_L  \xi -  \sigma  \pi (k+ 1/2) \right).
\end{equation}
\end{proposition}

% Due to the rapid decay  of $\Psi_a(\xi)$ as $\xi \rightarrow \infty$, the  inequality \eqref{Fourier_decay_CM:eqn} can be interpreted to say that the  $L^2$-energy of $\widehat{\phi_{L,k}}(\xi)$ is sharply concentrated near the two frequencies, where $\xi \sim \pm \pi (k+1/2)/ \delta_L$. \AM{is this detail required? it is taking some space :)}

We sketch the proof of Proposition \ref{basis:prop} -- for details, see \cite{israelmayeli2023acha}.
The system $\{\phi_{L,k}\}$ is a variant of the local sine basis for $L^2([0,1])$ (see \cite{CM,Weiss93}).  
We select cutoff functions $\theta_L$, with each $\theta_L$ supported on a neighborhood of the interval $L \in \cW$. We define a 
family, indexed by $L \in \cW$ and $k \in \N$,
\[
\phi_{L,k}(x) = c_L \theta_L(x) \sin(\pi(k+1/2)(x-x_L)/\delta_L),
\]
with $x_L$ the left endpoint of $L$, and $c_L$ a normalization constant picked so that $\| \phi_{L,k} \|  = 1$.
Coifman and  Meyer showed that the resulting family is an orthonormal basis for $L^2([0,1])$. The support condition in Proposition \ref{basis:prop} follows by the support condition for the cutoffs $\theta_L$. 
If the cutoff functions are suitably regular -- specifically, if they belong to the Gevrey class $G^{3/2}$ --  then the decay condition \eqref{Fourier_decay_CM:eqn} will be satisfied. This concludes our sketch of the proof of Proposition \ref{basis:prop}.

Observe that the $\phi_{L,k}$ do not technically form a wave packet system, but rather   a weighted combination of two wave packets with modulation frequencies of opposite signs. 
This follows from the identity $\sin(t) = (-i/2)(e^{it} - e^{-it})$.

\subsubsection{Tensor systems}
%

%\begin{figure}[ht]
%    \centering  \includegraphics[width=0.25\textwidth]{thinrectangles.png}
%    \caption{A decomposition of the square into rectangles $\bL$ in dimension $d=2$.}
%    \label{fig:cube-decomposition}
%\end{figure}

We construct an orthonormal basis for $L^2([0,1]^d)$. 
Let $\cW^{\otimes d} = \{ \bL = L_1 \times \cdots \times L_d : L_1,\cdots,L_d \in \cW\}$ be the set of all $d$-{\it fold Cartesian products} of intervals in $\cW$. Elements $\bL$ of $\cW^{\otimes d}$ are rectangular boxes in $[0,1]^d$. See Figure \ref{fig:thinrect} for a depiction of the family $\cW^{\otimes d}$.

For $\bL = L_1 \times \cdots \times L_d \in \cW^{\otimes d}$ and $\bk = (k_1,\cdots,k_d) \in \N^d$, we define $\psi_{(\bL,\bk)}$ in $L^2(\R^d)$ by
\[
\psi_{(\bL,\bk)}(x_1,\cdots,x_d) = \phi_{L_1, k_1}(x_1)\cdots\phi_{L_d,k_d}(x_d),
\]
where $\cB^1 = \{\phi_{L, k}\}_{(L,k) \in \cW \times \N}$ is as in Proposition \ref{basis:prop}.

Set $\cB^d = \{ \psi_{(\bL,\bk)} \}$. Then $\cB^d = (\cB^1)^{\otimes d}$ is the $d$-fold tensor product of the basis $\cB^1$ for $L^2([0,1])$. We identify the $d$-fold tensor power of $L^2([0,1])$ with $L^2([0,1]^d)$.

\begin{wrapfigure}{l}{0.27\textwidth} % 'r' for right; adjust width as needed
    \centering
    \vspace{-5pt}                     % Adjust vertical spacing if desired
\includegraphics[width=0.25\textwidth]{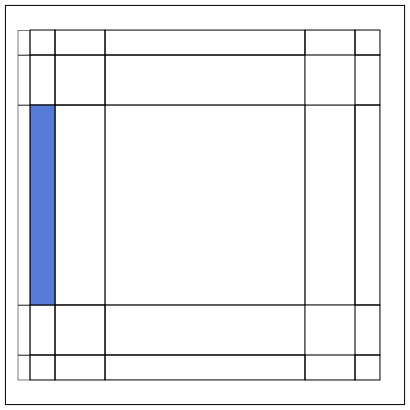}
    \caption{A decomposition of the square into rectangles $\bL$ in dimension $d=2$.}
    \label{fig:thinrect}
    \vspace{-5pt}                     % Adjust as needed to control text flow
\end{wrapfigure}

By properties of tensor products of Hilbert spaces, $\cB^d$ is an orthonormal basis for $L^2([0,1]^d)$. Each $\psi_{(\bL,\bk)}$ is supported on a neighborhood of the box $\bL$, and $\psi_{(\bL,\bk)}$ is the sum of $2^d$ many wave packets. Using the decay condition \eqref{Fourier_decay_CM:eqn}, one can then establish the following: 
\begin{proposition}\label{main:prop}
Let $Q = [0,1]^d$.  Let $S \subset B(0,1)$ be convex and coordinate-wise symmetric. Let $\epsilon \in (0,1/2)$ and $r \geq 1$ be given. The system $\cB^d = \{ \psi_{(\bL,\bk)}\}$ is an orthonormal basis for $L^2(Q)$, and admits a partition into three parts, $\cB^d = \cB^d_{low} \cup \cB^d_{res} \cup \cB^d_{hi}$, determined by $S, r, \epsilon$, satisfying:   
\begin{equation}\label{energy:est}
\sum_{\psi \in \cB^d_{hi}} \| \widehat{ \psi} \|_{L^2(S(r))}^2 + \sum_{\psi \in \cB^d_{low}} \| \widehat{ \psi}\|_{L^2(\R^d \setminus S(r))}^2 
\leq \epsilon^2/4
\end{equation}
with $\# \cB^d_{res} \leq C_d  \max \{ r^{d-1} \log(r/\epsilon)^{5/2}, \log(r/\epsilon)^{(5/2)d} )\}$. Here, $C_d$ is a dimensional constant determined solely by $d$.
\end{proposition}
For the proof of Proposition \ref{main:prop}, we apply Proposition 5.3 in \cite{israelmayeli2023acha} with $\delta = c_d \epsilon^2/r^d$. 
%for a small dimensional constant $c_d$.
For the proof of \eqref{energy:est}, see the estimates in Lemmas 5.4 and 5.6 of \cite{israelmayeli2023acha}.

\begin{proof}[Proof of Theorem \ref{mainthm:cube_convex}]
The operators $\mathcal{T}_{F,S} = B_S P_F B_S$ and 
  $P_F B_S P_F$ have identical eigenvalues
  (see the proof of Lemma 1 in \cite{Landau67}). The bound follows from Lemma \ref{conc_frame:lem} and Proposition \ref{main:prop}.
\end{proof}

{\bf Open Problem 2:} 
Moving beyond tensor systems, generalize the previous construction to prove an analogue of Proposition \ref{main:prop} in the setting when $Q$ and $S$ are Euclidean balls, i.e., $Q=B(0,1)$, $S(r) = B(0,r)$. That is, we ask to construct a system $\mathcal{B}^d$ that is a frame for $L^2(B(0,1))$ and such that for each $\epsilon > 0$, $\cB^d$ admits a decomposition  $\cB^d = \cB^d_{low} \cup \cB^d_{res} \cup \cB^d_{hi}$, satisfying \eqref{energy:est} with $\# \cB^d_{res} \leq C_d  r^{d-1} \log(r/\epsilon)^{J}$ for some exponent $J$.

By Lemma 3.1 in \cite{marceca2023}, this bound would recover the known spectral estimates on the localization operators $\mathcal{T}_{Q,S(r)}$ \cite{marceca2023,hughes2024eigenvalue}. Furthermore, we expect that the associated frame will have further use in signal processing tasks, such as signal recovery and denoising.

%\textcolor{blue}{
%We expect this frame to have important applications for developing signal manipulation and reconstruction
%methods relating to the processing of restrictions of bandlimited functions to the the unit ball.}

\printbibliography

\end{document}